\documentclass[leqno,final]{article}

\usepackage{graphicx}
\usepackage{amsmath,amstext,amssymb,amsthm}
\newtheorem{theorem}{Theorem}[section]
\newtheorem{lemma}[theorem]{Lemma}
\newtheorem{remark}{Remark}[section]
\newcommand{\abs}[1]{\left\vert#1\right\vert}

\newcommand{\norm}[1]{\left\Vert#1\right\Vert}
\newcommand{\norml}[2]{\left\Vert#1\right\Vert_{L^2(#2)}}

\newcommand{\normb}[1]{\left\Vert#1\right\Vert_{1,h}}
\newcommand{\normc}[1]{\left\Vert\hskip -0.8pt \left\vert #1 \right\vert\hskip -0.8pt\right\Vert_{1,h}}

\newcommand{\set}[1]{\left\{#1\right\}}
\newcommand{\av}[1]{\left\{#1\right\}}
\newcommand{\jm}[1]{\left[#1\right]}

\newcommand{\M}{\mathcal{M}}
\newcommand{\I}{\mathcal{I}}

\newcommand{\Pp}{\mathcal{P}_p}
\newcommand{\R}{\mathbb{R}}
\newcommand{\br}{\mathbf{r}}
\newcommand{\bn}{\mathbf{n}}

\newcommand{\db}{\displaybreak[0]}
\newcommand{\nn}{\nonumber}

\newcommand{\al}{\alpha}
\newcommand{\be}{\beta}

\newcommand{\ep}{\varepsilon}
\newcommand{\ga}{\gamma}
\newcommand{\Ga}{\Gamma}

\newcommand{\La}{\Lambda}
\newcommand{\na}{\nabla}

\newcommand{\Om}{\Omega}
\newcommand{\pa}{\partial}

\newcommand{\ze}{\zeta}

\newcommand{\rd}{\,\mathrm{d}}

\title{An unfitted $hp$-interface penalty finite element method for elliptic interface problems}
\author{
Haijun Wu
\thanks{Department of Mathematics, Nanjing University, Jiangsu,
210093, P.R. China. ({\tt hjw@nju.edu.cn}). The work of this author was
partially supported by the national basic research program of China
under grant 2005CB321701 and by the NSF of China grant 10971096.}
\and
Yuanming Xiao\thanks{Department of Mathematics, Nanjing University, Jiangsu,
210093, P.R. China. ({\tt xym@nju.edu.cn}).}}

\begin{document}
\date{}
\maketitle


\setcounter{page}{1}

\begin{abstract}
 An $hp$ version of interface penalty finite element method ($hp$-IPFEM) is proposed for  elliptic interface problems in two and three dimensions on unfitted meshes. Error estimates in broken $H^1$ norm, which are optimal  with respect to $h$ and suboptimal with respect to $p$ by half an order of $p$, are derived. Both symmetric and non-symmetric IPFEM are considered. Error estimates in $L^2$ norm are proved by the duality argument.
\end{abstract}




\section{Introduction}\label{sec-1}
Let $\Om=\Om_1\cup\Ga\cup\Om_2$ be a bounded and convex polygonal or polyhedral domain in $\R^d, d=2$ or $3$, where $\Om_1$ and $\Om_2$ are two subdomains of $\Om$ and $\Ga=\pa\Om_1\cap\pa\Om_2$ is a $C^2$-smooth interface (see Fig.~\ref{f1}). Consider the following elliptic interface problem:
\begin{equation}\label{eP}\left\{
\begin{aligned}
            & - \na\cdot\big(a(x) \na u\big)  =  f,\qquad   &\text{ in }\Om_1\cup\Om_2,\\
            & \jm{u}=g_D, \qquad  \jm{\big(a(x) \na u\big)\cdot\bn}=g_N,\qquad &\text{ on } \Ga, \\
            &  u =  0,\qquad &\text{ on } \pa\Om,
\end{aligned}\right.
\end{equation}
where $\jm{v}=v|_{\Om_1}-v|_{\Om_2}$ denotes the jump of $v$ across the interface $\Ga$, $\bn$ is the unit outward normal to the boundary of $\Om_1$, and $a(x)$ is bounded from below and above by some positive constants. Note that $a(x)$ is allowed to be discontinuous across the interface $\Ga$.
\begin{figure}
  \centering
  \includegraphics[scale=1]{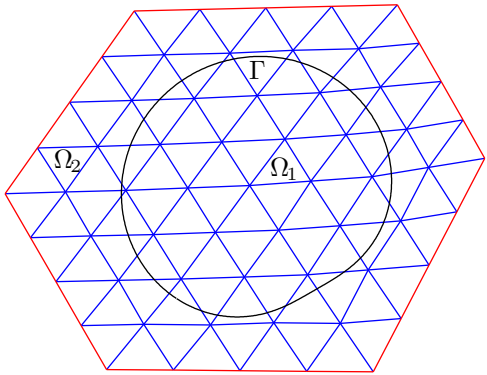}
\vskip -60pt  \caption{A sample domain $\Om$ and an unfitted mesh.}\label{f1}
\end{figure}

It is well known that the finite difference method (FDM) compared to the finite element method (FEM) is easy to implement since it discretizes the PDEs on simple meshes. But it is not easy to construct high order FDMs when the geometries are complicated. By contrast, the FEM provides a systematic way to design high order schemes for problems with complicated geometries, while the classical FEMs require that the interfaces and boundaries are fitted well enough by underlying meshes in order to achieve optimal convergence rates (cf. \cite{be87, cz98, lmwz10}). Therefore it has always been attractive to develop systematic ways for constructing high order numerical methods on simple unfitted meshes. Unfitted mesh techniques are particularly convenient for time-dependent problems when moving interfaces are involved.  One could just use the same mesh on the domain for different times instead of repeatedly remeshing the domain to fit the moving interfaces. Another motivation for developing unfitted methods is from the adaptive finite element methods. We know that in an adaptive finite element procedure, the meshes are locally refined repeatedly to equidistribute the error. But it is readily to encounter the ``inverted elements'', i.e., the elements with negative directional areas/volumes  (cf. \cite{shewchuk2002good}), near curved interfaces during interface-fitted local refinements, in particular, in the three dimensional case.     

Various finite difference schemes have been proposed on unfitted Cartesian grids, for example, the \emph{immersed boundary method} by Peskin \cite{peskin77}, the \emph{ghost fluid method} by Osher and his coworkers \cite{famo99}, the \emph{immersed interface method} by LeVeque and Li \cite{ll94}, and the \emph{matched interface and boundary method} by Zhou, et al. \cite{zzfw06}. These methods have the first order, second order or higher order truncation errors near the interfaces.  They have been generalized in various ways and applied to many problems \cite[etc]{lp04,lf05,tf03,lfk00,hloz97,hl99,lwcl03, yzw07}.

The existing finite element schemes on unfitted meshes are usually designed by proper modifications of the standard finite element methods near the interfaces.  The \emph{immersed finite element method} by Li, et al. \cite{llw03,gll08} modifies the basis functions for nodal points near the interfaces to satisfy the homogeneous jump conditions. The linear immersed finite element method is optimally  convergent in $H^{1}$ and  $L^{2}$ norms for two dimensional problems \cite{ckw10}. 
We refer to \cite{cxz09} for \emph{adaptive immersed interface finite element methods}
 for solving elliptic and Maxwell interface problems with singularities.
The \emph{multiscale finite element method} by Chu, Graham, and Hou \cite{cgh10} modifies the basis functions near the interface by solving appropriately designed subgrid problems. The method was proved, for the two dimensional case, to be of first-order accuracy in $H^{1}$ norm and of second-order accuracy in $L^{2}$ norm. Note that when the interface intersects an element in a straight line, the multiscale finite element method coincides with the linear immersed finite element method (cf. \cite{cgh10}).  
The \emph{penalty finite element method} proposed by Babu\v{s}ka \cite{b70}  modifies the bilinear form near the interface by penalizing the jump of the solution value across the interface. The method may use high order elements but is suboptimally convergent in $h$. The linear case of the method was analyzed again by Barrett and Elliott \cite{be87} and was proved to be optimally convergent in $h$ under more regularity assumptions on the exact solution.  We remark that the idea of adding penalty terms is widely used in the interior penalty Galerkin methods \cite[etc]{arnold82,abcm01,baker77,dd76,fw09,rwg99}. The \emph{unfitted finite element method} proposed by A.~Hansbo and P.~Hansbo \cite{hh02} can be viewed as an improvement of the linear version of the Babu\v{s}ka's method. This unfitted method was proved, for two dimensional elliptic interface problems, to be optimally $h$-convergent in both $H^{1}$ and $L^{2}$ norms and the error estimates are uniform with respect to the relative position of the interface to the unfitted grid. The key idea of their method is using weighted average flux across the interface in the bilinear form instead of using the arithmetic average one as in the standard interior penalty discontinuous Galerkin methods.  In their method, the weights along the intersection of the interface and a grid element are chosen according to the areas of the two parts of the element separated by the interface. More recently, Massjung \cite{m09} proposed an \emph{$hp$-unfitted discontinuous Galerkin method} for Problem~\eqref{eP} which uses weighted average flux across the interface and penalizes not only the jump condition on the solution value but also the jump condition on the flux. It was proved that, for the two dimensional case, the method converges in broken $H^1$ norm at an optimal rate with respect to $h$ and at a suboptimal rate with respect to $p$ by a factor of $p$. The flux penalty term is not essential in the convergence analysis but helpful for numerical stability (cf. \cite{m09}). We also refer to \cite{bastian2009unfitted, guyomarc2009discontinuous} for recent more studies on the discontinuous Galerkin methods for elliptic interface problems.

In this paper we propose an $hp$-interface penalty finite element method ($hp$-IPFEM) for the interface problem \eqref{eP}. In this method we penalize the jump conditions on the solution value as well as the flux across the interface and adopt the arithmetic average flux across the interface in the bilinear form. We provide a rigorous analysis to show that, for both two and three dimensional cases,  the error estimates in broken $H^{1}$ norm are optimal with respect to $h$ and suboptimal with respect to $p$ by a half order of $p$. Both symmetric and non-symmetric IPFEM are considered. Some $hp$-error estimates in $L^2$ norm are also derived by the duality argument.
We would like to mention that our tricks on dealing with discontinuities across the interface are common in the standard interior penalty discontinuous Galerkin methods. They are adopted by \cite{hh02,m09} except the arithmetic average one. Sure we do not add penalty terms on edges/faces of elements away from the interface as the standard interior penalty discontinuous Galerkin methods do.

For any intersection segment/patch $e$ between the interface and a grid element $K^e$, let
$K^e_i=K^e\cap\Om_i, i=1, 2, $ be the two parts of $K^e$ separated by the interface. The key idea of our analysis is the use of the following straight forward identity on $e$:
\begin{equation*}
\frac{v_{1}+v_{2}}{2}=v_{i}+\frac{(-1)^{i}}{2}(v_{1}-v_{2}), \quad\text{where}\quad v_i=(a\na v_h)|_{K^e_i}\cdot\bn, \quad i=1, 2.
\end{equation*}
By this we may bound the term of the arithmetic average flux (i.e. $\frac{v_{1}+v_{2}}{2}$) by any component in the average (say $v_i$, $i=$ either 1 or 2) and the jump of flux (i.e. $v_1-v_2$) which is controlled by the flux penalty term. This component is then estimated by a local inverse trace inequality on the corresponding part of the interface element. As a result, it suffices to prove the local inverse trace inequality on either $K^e_1$ or $K^e_2$ instead of on both of them as treated in \cite{m09}. Therefore we avoid discussing various relative positions of the interface to an element $K^e$ which are already complicated in two dimensional case since $K^e_1$ or $K^e_2$ may not be shape regular. This makes our analysis much simpler and easy to apply to the three dimensional case. To the best of our knowledge, our results give the first $hp$ a priori error estimates for interface problems in three dimensions and the best in two dimensions. 

The rest of our paper is organized as follows. We formulate the $hp$-interface penalty finite element methods in Section~\ref{sec-2} and list some preliminary lemmas in Section~\ref{sec-3}. The $H^{1}$- and $L^{2}$- error estimates of the symmetric interface penalty finite element methods are given in Section~\ref{sec-4}. The error estimates of the non-symmetric interface penalty finite element methods are given in Section~\ref{sec-5}. Section~\ref{sec-6} is devoted to the proofs of the local trace inequality and the local inverse trace inequality.

 Throughout the paper, $C$ is used to denote a generic positive constant
which is independent of $h$, $p$, and the penalty parameters. We also use the shorthand notation
$A\lesssim B$ and $B\gtrsim A$ for the inequality $A\leq C B$ and $B\geq CA$.
$A\eqsim B$ is for the statement $A\lesssim B$ and $B\lesssim A$. Note that the constants in this paper may depend on the jump of the coefficient $a(x)$ across the interface. Although it is an interesting topic to derive error estimates with explicit dependence on the jump of the coefficient $a(x)$, it may bury our basic idea in a more complicated analysis. We leave this issue to future work.

\section{Formulation of $hp$-interface penalty finite element methods}\label{sec-2}
To formulate our $hp$-IPFEMs, we need to introduce some notation. The space, norm and inner product notation used in this paper are all
standard, we refer to \cite{bs08,ciarlet78} for their precise
definitions.

Let $\set{\M_h}$ be a family of conforming, quasi-uniform, and regular
partitions of the domain $\Om$ (cf. \cite{bs08}) into triangles and parallelograms/tetrahedrons and parallelepipeds.  For any $K\in
\M_h$, we define $h_K:=\mbox{diam}(K)$. Let $h:=\max_{K\in\M_h} h_K$.
Then $h_K\eqsim h$. Note that any element $K\in\M_h$ is considered
as closed. Obviously, each partition $\M_h$ induces a partition, denoted by $\I_h$, of the interface $\Ga$ (cf. Fig.~\ref{f1}). Note that any interface segment/patch $e\in\I_h$ is either contained entirely in one element in $\M_h$ and has nonempty intersection with its interior,  or is the common edge/face of two neighboring elements in $\M_h$. For any  $e\in\I_h$, let $K^e\in\M_h$ be one of the element(s) containing $e$ and let
$K^e_i=K^e\cap\Om_i, i=1, 2$.

 Introduce the  ``energy" space
\begin{equation}\label{eV}
    V:= \set{v: \;v|_{\Om_i}=v_i,\; \text{ where } v_{i}\in H_0^1(\Om), \, v_i|_K\in H^2(K),\,\forall K\in\M_h, \,i=1, 2}.
\end{equation}
 For any
$K\in \M_h$, let $\Pp(K)$ denote the set of all polynomials
whose degrees in all variables (total degrees) do not exceed $p$ if $K$ is a triangle/tetrahedron, and  the set of all polynomials whose degrees in each variable (separate degrees) $\le p$ if $K$ is a parallelogram/parallelepiped. Denote by $U_h^p$ the $hp$-continuous finite element space, that is,
\begin{equation}\label{eUhp}
    U_h^p:=\set{v_h\in H_0^1(\Om) : \;v_h|_K\in \Pp(K),\;\forall K\in\M_h}.
\end{equation}
We define our interface penalty finite element approximation
space $V_h^p$ as
\begin{equation}\label{eVhp}
V_h^p:=\set{v_h: \;v_h|_{\Om_i}=v_{ih}, \; \text{ where } v_{ih}\in U_h^p,\;i=1, 2}.
\end{equation}
Clearly, $V_h^p\subset V\subset L^2(\Om)$.
But $V_h^p\not\subset H^1(\Om)$.

We also define the jump $\jm{v}$ and average $\av{v}$ of $v$ on the interface $\Ga$
 as
 \begin{equation}\label{eja}
    \jm{v}:=v|_{\Om_1}-v|_{\Om_2},\qquad \av{v}:=\frac{v|_{\Om_1}+v|_{\Om_2}}{2}.
 \end{equation}
 Testing the elliptic problem \eqref{eP} by any $v\in V$, using integration by parts, and using the identity $\jm{vw}=\av{v}\jm{w}+\jm{v}\av{w}$,  we obtain
\begin{align*}
    \sum_{i=1}^2 \int_{\Om_i}a\na u\cdot\na v-\int_\Ga \av{a\na u\cdot\bn}\jm{v}-\int_\Ga g_N\av{v}=\int_{\Om} fv
\end{align*}
Define the bilinear form $a_h(\cdot,\cdot)$ on $V\times V$:
\begin{align}
a_h(u,v) :=&\sum_{i=1}^2 \int_{\Om_i}a\na u\cdot\na v
-\sum_{e\in\I_h}\int_{e} \Big(\av{a\na u\cdot\bn}\jm{v}
+\be \jm{u}\av{a\na v\cdot\bn}\Big)\label{eah}\\
&+  J_0(u,v) +J_1(u,v),\nn \\
J_0(u,v):=&\sum_{e\in\I_h}\frac{\ga_0\, p^2}{h_{K^e}}\int_{e} \jm{u}\jm{v},\label{eJ0}\\
J_1(u,v):=&\sum_{e\in\I_h}\frac{\ga_1\,h_{K^e}}{ p^2}\int_e
\jm{a\na u\cdot\bn}\jm{a\na v\cdot\bn},\label{eJ1}
\end{align}
where $\be$ is a real number, $\ga_0$ and $\ga_1$
are nonnegative numbers to be specified later. Define further the  linear form $F_h(\cdot)$ on $V$:
\begin{align}
 F_h(v):=&\int_{\Om} fv+\int_\Ga g_N\av{v}-\be\int_{\Ga} g_D\av{a\na v\cdot\bn}+J_D(v)+J_N(v),\\
 J_D(v):=&\sum_{e\in\I_h}\frac{\ga_0\, p^2}{h_{K^e}}\int_e g_D\jm{v},\label{eJD}\\
 J_N(v):=&\sum_{e\in\I_h}\frac{\ga_1\,h_{K^e}}{ p^2}\int_e g_N\jm{a\na v\cdot\bn}.\label{eJN}
\end{align}

It is easy to check that the solution $u$ to the problem \eqref{eP} satisfies the following formulation:
\begin{equation}\label{evp}
a_h(u,v) =F_h(v), \qquad\forall v\in V.
\end{equation}

  We are now ready to define
our $hp$-interface penalty finite element methods inspired by the formulation
\eqref{evp}:
Find $u_h\in V_h^p$ such that
\begin{equation}\label{eipfem}
a_h(u_h, v_h) =F_h(v_h), \qquad\forall  v_h\in V_h^p.
\end{equation}
\begin{remark}
\begin{enumerate}
  \item[{\rm (i)}] The above tricks on dealing with the discontinuities are from  the interior penalty discontinuous or continuous Galerkin methods (see, e.g., \cite{arnold82,fw09,be07}). When $\be=1$,  $a_h(\cdot,\cdot)$ is symmetric, that is, $a_h(v,w)=a_h(w,v)$, the method is termed as the symmetric interface penalty finite element method (SIPFEM), which corresponds to the symmetric interior penalty Galerkin method \cite{arnold82,w78}.
On the other hand, when $\be\neq 1$,
$a_h(\cdot,\cdot)$ is non-symmetric. In particular, when $\be =-1$, the method is called  the non-symmetric interface penalty finite element method (NIPFEM), which would correspond to the non-symmetric interior penalty Galerkin method  studied in \cite{rwg99}.
In this paper, for the ease of presentation, we only consider the case $\be=1$ and $\be=-1$,
nevertheless the ideas of this paper also apply to the  case $\be\neq \pm1$.
  \item[{\rm (ii)}] The unfitted finite element method given in \cite{hh02} does not include the penalty term $J_{1}$ and uses a weighted average $\av{v}=\kappa_{1}v|_{\Om_{1}}+\kappa_{2}v|_{\Om_{2}}$ instead of the arithmetic one in \eqref{eja}, where $\kappa_{i}|_{e}=\abs{K^e_i}/\abs{K}$. Optimal error estimates in both $H^{1}$ and $L^{2}$ norms were derived in \cite{hh02} for the linear element ($p=1$) in two dimensions.
  \item[{\rm (iii)}]  Massjung \cite{m09} proposed an $hp$-unfitted discontinuous Galerkin method for Problem~\eqref{eP} which uses weighted averages similar as in \cite{hh02} and penalizes both the jump of the solution value and that of the flux. It was proved that, for the two dimensional case, the method converges in broken $H^1$ norm at an optimal rate with respect to $h$ and at a suboptimal rate with respect to $p$ by a factor of $p$. 
 \end{enumerate}
 \end{remark}

Next, we introduce the following broken (semi-)norms on the space $V$:
\begin{align}
\normb{v}&:=\bigg(\sum_{i=1}^2\norml{a^{1/2}\na v}{\Om_i}^2+ \sum_{e\in\I_h} \frac{\ga_0\, p^2}{h_{K^e}}\norml{\jm{v}}{e}^2\label{enorma}\\
&\hskip 108pt+\sum_{e\in\I_h}\frac{\ga_1\,h_{K^e}}{ p^2}\norml{\jm{a\na v\cdot\bn}}{e}^2 \bigg)^{1/2},\nn  \db\\
\normc{v}&:=\bigg(\normb{v}^2+\sum_{e\in\I_h}
\frac{h_{K^e}}{\ga_0\, p^2}\norml{\av{a\na v\cdot
\bn}}{e}^2\bigg)^{1/2}. \label{enormb}
\end{align}

\section{Some preliminary lemmas}\label{sec-3}

Recall that, for any interface segment/patch $e\in\I_h$, $K^e\in\M_h$ is  an element containing $e$ and
$K^e_i=K^e\cap\Om_i, i=1, 2$. We have the
following lemma, which gives the trace and inverse trace
inequalities on $K^e$, whose proof will be given in
Section~\ref{sec-6}.
\begin{lemma}\label{ltrace} There exists a positive constant $h_0$ depending only on the interface $\Ga$ and the shape regularity of the meshes, such
that for all $h\in (0, h_0]$ and any interface segment/patch $e\in\I_h$, the following estimates hold for either $i_e=1$ or $i_e=2$:
  \begin{enumerate}
    \item[{\rm(i)}] For any $v\in H^1(K^e_{i_e})$,
    \begin{equation}
    \norml{v}{e}\lesssim h_{K^e}^{-1/2}\norml{v}{K^e_{i_e}}+\norml{v}{K^e_{i_e}}^{1/2}\norml{\na v}{K^e_{i_e}}^{1/2}.\label{tr1}
    \end{equation}
    \item[{\rm(ii)}] For any $v_h\in \Pp(K^e)$,
    \begin{equation}
    \norml{v_h}{e}\lesssim\frac{p}{h_{K^e}^{1/2}}\norml{v_h}{K^e_{i_e}}.\label{tr2}
    \end{equation}
  \end{enumerate}
\end{lemma}
We remark that, if $e\subset\pa K^e$, then the result (i) in the above lemma is a direct consequence of  the standard local trace inequality, and the estimate in (ii) is already known (cf. \cite{be07}). 

The error analysis relies on the following well-known $hp$ approximation properties (cf. \cite{bs87,guo06,gs07}):
\begin{lemma}\label{lapp0} Suppose $w\in H^s(\Om)\cap H_0^1(\Om), s\ge 2$. Let $\mu= \min\set{p+1,s}$.
  \begin{enumerate}
\item[{\rm (i)}] There exists a piecewise polynomial $\tilde w_h$,  $\tilde w_h |_K\in \Pp(K)$, such that, for $ j=0,1,\cdots,s$,
\begin{equation*}
\norm{w-\tilde w_h}_{H^j(K)}\lesssim\frac{h^{\mu-j}}{p^{s-j}}\norm{w}_{H^{s}(K)},\quad \forall K\in \M_h.
 \end{equation*}
\item[{\rm (ii)}]  There exists $\hat w_h\in U_h^p$ such that
 \begin{equation*}
   \norm{w-\hat w_h}_{H^j(\Om)}\lesssim \frac{h^{\mu-j}}{p^{s-j}}\norm{w}_{H^{s}(\Om)}, \quad j=0,1.
 \end{equation*}
\end{enumerate}
Here the invisible constants in the above two inequalities depend on $s$ but are independent of  $h$ and $p$. $U_h^p$ is the $hp$-continuous finite element space defined in \eqref{eUhp}.
\end{lemma}

The following lemma is from the  standard Sobolev extension property (cf. \cite{adm75}).
\begin{lemma}\label{lext} Let $s\ge 2$. There exist two extension operators $E_1: H^s(\Om_1)\mapsto H^s(\Om)\cap H_0^1(\Om)$  and $E_2: \set{w\in H^s(\Om_2):\; w|_{\pa\Om}=0}\mapsto H^s(\Om)\cap H_0^1(\Om)$ such that
  \[(E_i w)|_{\Om_i}=w \quad\text{ and }\quad \norm{E_i w}_{H^s(\Om)}\lesssim\norm{w}_{H^s(\Om_i)},\quad  i=1, 2,\]
where $w\in H^s(\Om_1)$ for $i=1$ and $w\in H^s(\Om_2),\, w|_{\pa\Om}=0$ for $i=2$ (cf. Fig.~\ref{f1}).
\end{lemma}

\section{Error estimates for the symmetric interface penalty finite element method}\label{sec-4}
In this section we derive the  $H^1$- and $L^2$- error estimates for the $hp$-SIPFEM (i.e. the case $\beta=1$).
The following lemma gives the continuity and coercivity of the
bilinear form $a_h(\cdot,\cdot)$ for the SIPFEM.
 \begin{lemma}\label{lbil} We have
 \begin{align}\label{econt}
    \abs{a_h(v,w)}\le 2\normc{v}\normc{w},\quad\forall\, v,w\in V.
 \end{align}
For any $0<\ga_1\lesssim 1$ and $0<\ep<1$, there exists a constant $\al_{0,\ep}$ independent of $h$, $p$, and the penalty parameters such that, if $\ga_0\ge\al_{0,\ep}\big/\ga_1$, then, for $0<h\le h_{0}$,
\begin{equation}\label{ecoer}
     a_h(v_h,v_h)\ge (1-\ep)\normc{v_h}^2,\quad\forall\, v_h\in V_h^p,
\end{equation}
  where the constant $h_{0}$ is from Lemma~\ref{ltrace}.
 \end{lemma}
\begin{proof}
  \eqref{econt} is a  direct consequence of the definitions \eqref{eah}--\eqref{eJ1}, \eqref{enorma}--\eqref{enormb}, and the Cauchy-Schwarz inequality. It remains to prove \eqref{ecoer}. We have,
  \begin{align}\label{elbil1}
    a_h(v_h,v_h)=&\sum_{i=1}^2\norml{a^{1/2}\na v_h}{\Om_i}^2-2\sum_{e\in\I_h}\int_e \av{a\na v_h\cdot\bn}\jm{v_h}\\
    &+\sum_{e\in\I_h} \bigg(\frac{\ga_0\, p^2}{h_{K^e}}\norml{\jm{v_h}}{e}^2
    +\frac{\ga_1\,h_{K^e}}{ p^2}\norml{\jm{a\na v_h\cdot\bn}}{e}^2\bigg)\nn\\
=&\normc{v_h}^2-\sum_{e\in\I_h} \frac{h_{K^e}}{\ga_0\, p^2}\norml{\av{a\na v_h\cdot\bn}}{e}^2\nn\\
&\hskip 37pt -2\sum_{e\in\I_h}\int_e \av{a\na v_h\cdot
\bn}\jm{v_h}\nn.
  \end{align}
  It is clear that, for any $e\in\I_h$,
 \begin{equation}\label{ekey}
    \av{a\na v_h\cdot\bn}|_{e}=(a\na v_h)|_{K^e_{i_e}}\cdot\bn+\frac{(-1)^{i_e}}{2}\jm{a\na v_h\cdot\bn}|_{e},
 \end{equation}
  where $i_e$ is specified in Lemma~\ref{ltrace}. From Lemma~\ref{ltrace} (ii), there exists a constant $C_1$ such that,
  \begin{align}\label{elbil2}
    \norml{(a\na v_h)|_{K^e_{i_e}}\cdot\bn}{e}&\lesssim \frac{p}{h_{K^e}^{1/2}} \norml{\na v_h}{K^e_{i_e}}\le C_1\frac{p}{h_{K^e}^{1/2}} \norml{a^{1/2}\na v_h}{K^e_{i_e}}.
  \end{align}
Therefore, from \eqref{ekey} and \eqref{elbil2},
\begin{align}\label{elbil3}
   &\sum_{e\in\I_h}\frac{h_{K^e}}{\ga_0\, p^2} \norml{\av{a\na v_h\cdot\bn}}{e}^2\\
   &\le 2\sum_{e\in\I_h}\frac{h_{K^e}}{\ga_0\, p^2} \Big(\norml{(a\na v_h)|_{K^e_{i_e}}\cdot\bn}{e}^2+\frac14\norml{\jm{a\na v_h\cdot\bn}}{e}^2\Big)\nn\\
   &\le \sum_{e\in\I_h}\frac{2 C_1^2}{\ga_0}\norml{a^{1/2}\na v_h}{K^e_{i_e}}^2+ \sum_{e\in\I_h}\frac{h_{K^e}}{2 \ga_0\, p^2} \norml{\jm{a\na v_h\cdot\bn}}{e}^2\nn\\
   &\le \max\Big(\frac{2 C_1^2}{\ga_0}, \frac{1}{2\ga_0\ga_1}\Big)\normb{v_h}^2.\nn
\end{align}
On the other hand,
  \begin{align}\label{elbil4}
    &2\sum_{e\in\I_h}\int_e \av{a\na v_h\cdot\bn}\jm{v_h}\\
    &\le 2\sum_{e\in\I_h}\Big(\norml{(a\na v_h)|_{K^e_{i_e}}\cdot\bn}{e}+ \frac12\norml{\jm{a\na v_h\cdot\bn}}{e}\Big)\norml{\jm{v_h}}{e}\nn\\
    &\le \sum_{e\in\I_h}\bigg(2 C_1\norml{a^{1/2}\na v_h}{K^e_{i_e}}\frac{p}{h_{K^e}^{1/2}} \norml{\jm{v_h}}{e}\nn\\
    &\qquad\qquad\quad+\norml{\jm{a\na v_h\cdot\bn}}{e} \norml{\jm{v_h}}{e}\bigg)\nn\\
    &\le\sum_{e\in\I_h}\bigg( \frac{4C_1^2}{\ga_0\ep}\norml{a^{1/2}\na v_h}{K^e_{i_e}}^2
    +\frac{\ep}{4}\frac{\ga_0\, p^2}{h_{K^e}} \norml{\jm{v_h}}{e}^2\nn\\
    &\qquad\qquad\quad+\frac{1}{\ga_0\ep}\frac{h_{K^e}}{ p^2}\norml{\jm{a\na v_h\cdot\bn}}{e}^2
    +\frac{\ep}{4}\frac{\ga_0 p^2}{h_{K^e}} \norml{\jm{v_h}}{e}^2
    \bigg)\nn\\
    &\le\max\Big(\frac{4C_1^2}{\ga_0\ep}, \frac{1}{\ga_0\ga_1\ep}, \frac{\ep}{2}\Big)\normb{v_h}^2.\nn
  \end{align}
By combining \eqref{elbil1}, \eqref{elbil3}, and \eqref{elbil4}, we conclude that \eqref{ecoer} holds. This completes the proof of the lemma.
\end{proof}

The following lemma is an analogue of the Cea's lemma for finite element methods.
 \begin{lemma}\label{lcea}
 There exists a constant $\al_{0}$ independent of $h$, $p$, and the penalty parameters such that for $0<\ga_1\lesssim 1$, $\ga_0\ge\al_0\big/\ga_1$, and $0<h\le h_{0}$, the following error estimate holds.
 \begin{align*}
    \normc{u-u_h}\lesssim\inf_{z_h\in V_h^p}\normc{u-z_h},
 \end{align*}
   where the constant $h_{0}$ is from Lemma~\ref{ltrace}.
\end{lemma}
 \begin{proof} From \eqref{evp} and \eqref{eipfem}, we have the following Galerkin orthogonality
 \begin{equation}\label{eorth}
    a_h(u-u_h,v_h)=0,\quad\forall v_h\in V_h^p.
 \end{equation}
 For any $z_h\in V_h^p$, let $\eta_h=u_h-z_h$.
  By letting $\ep=1/2$ and in Lemma~\ref{lbil}, denoting by $\al_0=\al_{0,1/2}$, and using \eqref{eorth}, we have,
  for $0<\ga_1\lesssim 1$ and $\ga_0\ge\al_0\big/\ga_1$,
 \begin{align*}
    \normc{\eta_h}^2\le &2 a_h(\eta_h,\eta_h)=2 a_h(u-z_h,\eta_h)
    \le 4\normc{u-z_h}\normc{\eta_h},
 \end{align*}
 and hence,
\begin{equation*}
    \normc{u_h-z_h}\le 4\normc{u-z_h},
\end{equation*}
which implies that
 \begin{align*}
   \normc{u-u_h}=\normc{u-z_h+z_h-u_h}\le 5\normc{u-z_h}.
 \end{align*}
 This completes the proof of the lemma.
 \end{proof}

  The following lemma gives some approximation properties of the space $V_h^p$.
  \begin{lemma}\label{lapp} Let $s\ge 2$ be an integer and let $\mu= \min\set{p+1,s}$. Suppose the solution to the elliptic interface problem \eqref{eP} satisfies $u|_{\Om_i}\in H^s(\Om_i), i=1,2$. Then there exists $\hat u_h\in V_h^p$ such that, for $0<h\le h_{0}$,
 \begin{align}
    &\norml{u-\hat u_h}{\Om}\lesssim \frac{h^\mu}{p^s}\bigg(\sum_{i=1}^2\norm{u}_{H^s(\Om_i)}^2\bigg)^{1/2},\label{elapp1}\\
    &\normc{u-\hat u_h}\lesssim\Big(\frac{1}{p}+\ga_0+\frac{1}{\ga_0\,p}+\frac{\ga_1}{p}\Big)^{1/2}\frac{h^{\mu-1}}{p^{s-3/2}}
    \bigg(\sum_{i=1}^2\norm{u}_{H^s(\Om_i)}^2\bigg)^{1/2},\label{elapp2}
 \end{align}
    where the constant $h_{0}$ is from Lemma~\ref{ltrace}.
 \end{lemma}
 \begin{proof}
   Let $u_i=E_i u$, where $E_i, i=1, 2$ are the extension operators from Lemma~\ref{lext}. Then
   \begin{equation}\label{elapp3}
 u_i\in H^s(\Om)\cap H_0^1(\Om),\qquad u_i|_{\Om_i}=u|_{\Om_i}, \qquad \norm{u_i}_{H^s(\Om)}\lesssim\norm{u}_{H^s(\Om_i)}.
\end{equation}
From Lemma~\ref{lapp0}(ii) and \eqref{elapp3}, there exist $\hat u_{ih}\in U_h^p, i=1, 2$ such that
\begin{equation}\label{elapp4}
    \norm{u_i-\hat u_{ih}}_{H^j(\Om)}\lesssim \frac{h^{\mu-j}}{p^{s-j}}\norm{u}_{H^{s}(\Om_i)}, \quad j=0,1.
\end{equation}
Define $\hat u_h\in V_h^p$ by $\hat u_h|_{\Om_i}=\hat u_{ih}|_{\Om_i}, i=1, 2.$ It is clear that \eqref{elapp1} holds.

Denote by $\ze=u-\hat u_h$ and by $\ze_i=u_i-\hat u_{ih}, i=1, 2$. Obviously, $\ze|_{\Om_i}=\ze_i|_{\Om_i}, i=1, 2$. Next we estimate each term in $\normc{\ze}$ (cf. \eqref{enorma}--\eqref{enormb}). Clearly,
\begin{equation}\label{elapp6}
    \sum_{i=1}^2\norml{a^{1/2}\na \ze}{\Om_i}^2\lesssim \frac{h^{2\mu-2}}{p^{2s-2}}\bigg(\sum_{i=1}^2\norm{u}_{H^s(\Om_i)}^2\bigg).
\end{equation}
From Lemma~\ref{ltrace} (i) and \eqref{elapp4},
\begin{align}\label{elapp7}
     &\sum_{e\in\I_h} \frac{\ga_0\, p^2}{h_{K^e}}\norml{\jm{\ze}}{e}^2
     \lesssim \sum_{e\in\I_h} \frac{\ga_0\, p^2}{h_{K^e}}\sum_{i=1}^2\norml{\ze_i}{e}^2\\
     &\lesssim\sum_{e\in\I_h}\frac{\ga_0\, p^2}{h_{K^e}} \sum_{i=1}^2\Big(h_{K^e}^{-1}\norml{\ze_i}{K^e}^2+\norml{\ze_i}{K^e}\norml{\na \ze_i}{K^e}\Big)\nn\\
     &\lesssim \frac{\ga_0 p^2}{h} \sum_{i=1}^2\Big(h^{-1}\norml{\ze_i}{\Om}^2+\norml{\ze_i}{\Om}\norml{\na \ze_i}{\Om}\Big)\nn\\
     &\lesssim \ga_0 \frac{h^{2\mu-2}}{p^{2s-3}}\bigg(\sum_{i=1}^2\norm{u}_{H^s(\Om_i)}^2\bigg).\nn
\end{align}
To estimate the remaining terms, we estimate $\norml{(a\na \ze)|_{\Om_i}\cdot
\bn}{e}^2, i=1, 2, e\in\I_h.$ From Lemma~\ref{lapp0} (i),
there exist piecewise polynomials $\tilde u_{ih}$, $\tilde u_{ih}|_K\in \Pp(K)$, such that
\[\norm{u_i-\tilde u_{ih}}_{H^j(K)}\lesssim\frac{h^{\mu-j}}{p^{s-j}}\norm{u_i}_{H^{s}(K)},\quad\forall\, j=0, 1, 2, \quad i=1, 2, \quad K\in\M_h.\]
Therefore from Lemma~\ref{ltrace},
\begin{align*}
    \norml{(a\na \ze)|_{\Om_i}\cdot\bn}{e}^2&\lesssim \norml{\na (u_i-\tilde u_{ih})}{e}^2+\norml{\na (\tilde u_{ih}-\hat u_{ih})}{e}^2\\
    &\lesssim h_{K^e}^{-1} \norml{\na (u_i-\tilde u_{ih})}{K^e}^2\\
    &\qquad\qquad +\norml{\na(u_i-\tilde u_{ih})}{K^e}\abs{u_i-\tilde u_{ih}}_{H^2(K^e)} \\
    &\qquad\qquad +\frac{p^2}{h_{K^e}}\norml{\na (\tilde u_{ih}-u_i+u_i-\hat u_{ih})}{K^e}^2 \\
    &\lesssim \frac{h^{2\mu-3}}{p^{2s-4}}\norm{u_i}_{H^{s}(K^e)}^2 +\frac{p^2}{h}\norml{\na (u_i-\hat u_{ih})}{K^e}^2. \\
\end{align*}
It follows from \eqref{elapp3} and \eqref{elapp4} that
\[\sum_{e\in\I_h}\norml{(a\na \ze)|_{\Om_i}\cdot\bn}{e}^2\lesssim\frac{h^{2\mu-3}}{p^{2s-4}} \norm{u}_{H^{s}(\Om_i)}^2.\]
Thus,
\begin{align}\label{elapp8}
    &\sum_{e\in\I_h}\bigg(\frac{\ga_1\,h_{K^e}}{ p^2}\norml{\jm{a\na \ze\cdot\bn}}{e}^2+\frac{h_{K^e}}{\ga_0\, p^2}\norml{\av{a\na \ze\cdot\bn}}{e}^2 \bigg)\\
    &\lesssim \bigg(\ga_1+\frac{1}{\ga_0}\bigg)\frac{h^{2\mu-2}}{p^{2s-2}}\bigg(\sum_{i=1}^2\norm{u}_{H^s(\Om_i)}^2\bigg)\nn
\end{align}
Then \eqref{elapp2} follows by combining \eqref{elapp6}--\eqref{elapp8}. This completes the proof of the lemma.
 \end{proof}

By combining Lemma~\ref{lcea} and Lemma~\ref{lapp}, we have the following theorem  which gives  error estimate in the  norm $\normc{\cdot}$ for the SIPFEM. The proof is straightforward and is omitted.
\begin{theorem}\label{t1}
  Let $s\ge 2$ be an integer and let $\mu= \min\set{p+1,s}$. Suppose the solution to the elliptic interface problem \eqref{eP} satisfies $u|_{\Om_i}\in H^s(\Om_i), i=1,2$. Then there exists a constant $\al_{0}$  independent of $h$, $p$, and the penalty parameters such that for $0<\ga_1\lesssim 1$ and $\ga_0\ge\al_0\big/\ga_1$, the following error estimate holds:
  \begin{equation*}
    \normc{u-u_h}\lesssim \ga_0^{1/2}\frac{h^{\mu-1}}{p^{s-3/2}}
    \bigg(\sum_{i=1}^2\norm{u}_{H^s(\Om_i)}^2\bigg)^{1/2}, \quad \forall\, 0<h\le h_{0},
  \end{equation*}
  where the constant $h_{0}$ from Lemma~\ref{ltrace} depends only on the interface $\Ga$ and the shape regularity of the meshes.
\end{theorem}

Next we derive the $L^2$-error estimate by using the Nitsche's duality
argument (cf. \cite{ciarlet78}). Consider the following auxiliary problem:
\begin{equation}\label{ePD}\left\{
\begin{aligned}
            & - \na\cdot\big(a(x) \na w\big)  =  u-u_h,   &\text{ in }\Om_1\cup\Om_2,\\
            & \jm{w}=0, \quad  \jm{\big(a(x) \na w\big)\cdot\bn}=0, &\text{ on } \Ga, \\
            &  w =  0, &\text{ on } \pa\Om,
\end{aligned}\right.
\end{equation}
 It can be shown that $w$ satisfies (cf. \cite{b70})
\begin{equation}\label{ewsta}
\norm{w}_{H^2(\Om_1)}+\norm{w}_{H^2(\Om_2)}\lesssim \norml{u-u_h}{\Om}.
\end{equation}

 \begin{theorem}\label{t2} Under the conditions of Theorem~\ref{t1}, there holds the following estimate for the SIPFEM.
 \begin{equation*}
    \norml{u-u_h}{\Om}\lesssim\ga_0\frac{h^{\mu}}{p^{s-1}}
    \bigg(\sum_{i=1}^2\norm{u}_{H^s(\Om_i)}^2\bigg)^{1/2},\quad \forall\,0<h\le h_{0},
  \end{equation*}
  where the constant $h_{0}$ is from Lemma~\ref{ltrace}.
 \end{theorem}
\begin{proof}Let $\eta=u-u_h$. Testing \eqref{ePD} by $\eta$ and using \eqref{eorth} we get
\begin{align}\label{et21}
    \norml{\eta}{\Om}^2=a_h(w,\eta)=a_h(\eta,w)=a_h(\eta,w-\hat w_h),
\end{align}
where $\hat w_h\in V_h^p$ satisfies the estimate (cf. Lemma~\ref{lapp})
\begin{equation}\label{et22}
    \normc{w-\hat w_h}\lesssim \ga_0^{1/2}\frac{h}{p^{1/2}}\sum_{i=1}^2\norm{w}_{H^2(\Om_i)}.
\end{equation}
  Therefore from \eqref{econt} and \eqref{ewsta},
\begin{align*}
    \norml{\eta}{\Om}^2\le 2\normc{w-\hat w_h}\normc{\eta}\lesssim \ga_0^{1/2}\frac{h}{p^{1/2}}\norml{\eta}{\Om}\normc{\eta},
\end{align*}
that is $\norml{\eta}{\Om}\lesssim \ga_0^{1/2}\frac{h}{p^{1/2}}\normc{\eta}$, which completes the proof of Theorem~\ref{t2}  by using Theorem~\ref{t1}.
\end{proof}

\section{Error estimates for the non-symmetric interface penalty finite element methods}\label{sec-5}
In this section we derive the  $H^1$- and $L^2$- error estimates for the $hp$-NIPFEM (i.e. the case $\beta=-1$).
The following lemma gives the continuity and coercivity of the
bilinear form $a_h(\cdot,\cdot)$ for the NIPFEM.
 \begin{lemma}\label{lnbil} We have, for the NIPFEM ($\beta=-1$),
 \begin{align}\label{encont}
    \abs{a_h(v,w)}\le 2\normc{v}\normc{w},\quad\forall\, v,w\in V.
 \end{align}
\begin{equation}\label{encoer}
     a_h(v_h,v_h)\gtrsim \frac{\ga_0\ga_1}{\ga_0\ga_1+\max\big(\ga_1,1\big)}\normc{v_h}^2,\quad\forall\, v_h\in V_h^p,\quad\forall\,0<h\le h_{0},
\end{equation}
   where the constant $h_{0}$ is from Lemma~\ref{ltrace}.
 \end{lemma}
\begin{proof}
  \eqref{encont} is a  direct consequence of the definitions \eqref{eah}--\eqref{eJ1}, \eqref{enorma}--\eqref{enormb}, and the Cauchy-Schwarz inequality. From \eqref{elbil3},
  \[\normc{v_h}^2\le \Big(1+\max\Big(\frac{2 C_1^2}{\ga_0}, \frac{1}{2\ga_0\ga_1}\Big)\Big)\normb{v_h}^2.\]
  Then \eqref{encoer} follows from the identity $a_h(v_h,v_h)=\normb{v_h}^2$. This completes the proof of the lemma.
\end{proof}

Following the argument of Theorem~\ref{t1} we have the following   error estimate in the  norm $\normc{\cdot}$ for the $hp$-NIPFEM. The proof is omitted.
\begin{theorem}\label{t3}
  Let $s\ge 2$ be an integer and let $\mu= \min\set{p+1,s}$. Suppose the solution to the elliptic interface problem \eqref{eP} satisfies $u|_{\Om_i}\in H^s(\Om_i), i=1,2$. For $0<\ga_1\lesssim 1$ and $\ga_0\gtrsim 1\big/\ga_1$, the following error estimate holds.
  \begin{equation*}
    \normc{u-u_h}\lesssim  \ga_0^{1/2}\frac{h^{\mu-1}}{p^{s-3/2}}
    \bigg(\sum_{i=1}^2\norm{u}_{H^s(\Om_i)}^2\bigg)^{1/2},\quad \forall\,0<h\le h_{0},
  \end{equation*}
  where the constant $h_{0}$ is from Lemma~\ref{ltrace}.
\end{theorem}

The following theorem gives the $L^2$-error estimate for the $hp$-NIPFEM.
 \begin{theorem}\label{t4} Under the conditions of Theorem~\ref{t3}, there holds the following estimate for the NIPFEM.
 \begin{equation*}
    \norml{u-u_h}{\Om}\lesssim  \Big( \ga_0\frac{h^{\mu}}{p^{s-1}}+\frac{h^{\mu-1/2}}{p^{s-1/2}}\Big)
    \bigg(\sum_{i=1}^2\norm{u}_{H^s(\Om_i)}^2\bigg)^{1/2},\quad \forall\,0<h\le h_{0},
  \end{equation*}
  where the constant $h_{0}$ is from Lemma~\ref{ltrace}.
 \end{theorem}
\begin{proof} The proof is similar to that of Theorem~\ref{t2}. So we only sketch the differences here. \eqref{et21} becomes
\begin{align}\label{et41}
    \norml{\eta}{\Om}^2=a_h(w,\eta)=&a_h(\eta,w)-2\sum_{e\in\I_h}\int_e\jm{\eta}\av{a\na w\cdot\bn}\\
    =&a_h(\eta,w-\hat w_h)-2\sum_{e\in\I_h}\int_e\jm{\eta}\big(a\na w\cdot\bn\big),\nn
\end{align}
where $\eta=u-u_h$, $w$ is defined in \eqref{ePD}, and $\hat w_h\in V_h^p$ satisfies the estimate in \eqref{et22}.
 From \eqref{encont},\eqref{ewsta}, and \eqref{et22},
\begin{align}\label{et42}
    a_h(\eta,w-\hat w_h)\lesssim \ga_0^{1/2}\frac{h}{p^{1/2}}\norml{\eta}{\Om}\normc{\eta},
\end{align}
The last term in \eqref{et41} is bounded by using the Cauchy-Schwarz inequality and the trace
inequality:
\begin{align}\label{et43}
    -2\sum_{e\in\I_h}\int_e\jm{\eta}&\big(a\na w\cdot\bn\big)
    \le 2\sum_{e\in\I_h}\norml{\jm{\eta}}{e}\norml{a\na w\cdot\bn}{e}\\
    &\lesssim J_0(\eta,\eta)^{1/2}\bigg(\sum_{e\in\I_h}\frac{ h_{K^e}}{\ga_0\, p^2}\norml{\na w}{e}^2\bigg)^{1/2}\nn\\
    &\lesssim \frac{h^{1/2}}{\ga_0^{1/2} p}J_0(\eta,\eta)^{1/2}\norml{\na w}{\Ga}\nn\\
    &\lesssim \frac{h^{1/2}}{\ga_0^{1/2} p}J_0(\eta,\eta)^{1/2}\Big(\norml{\na w}{\Om_1}\norm{\na w}_{H^1(\Om_1)}\Big)^{1/2}\nn\\
     &\lesssim \frac{h^{1/2}}{\ga_0^{1/2} p}\normc{\eta}\norml{\eta}{\Om},\nn
\end{align}
where we have used \eqref{ewsta} to derive the last inequality. By combining \eqref{et41}--\eqref{et43} we conclude that
\[\norml{\eta}{\Om}\lesssim \Big( \ga_0^{1/2}\frac{h}{p^{1/2}}+\frac{h^{1/2}}{\ga_0^{1/2} p}\Big)\normc{\eta},\]
 which completes the proof of Theorem~\ref{t4}  by using Theorem~\ref{t3}.
\end{proof}

\begin{remark} Here is a brief comparison between the SIPFEM and the NIPFEM:
\begin{itemize}
\item The stiffness matrix of the SIPFEM is symmetric while that of the NIPFEM is not.
\item Both methods require that the penalty parameters satisfy $\ga_0\ga_1\ge C$ but  $C$ is a problem dependent constant for the SIPFEM while $C$ can be an arbitrary positive constant for the NIPFEM.
\item The $L^2$-error estimate for the SIPFEM is optimal with respect to $h$ while that for the NIPFEM is not. 
\end{itemize}
\end{remark}

\section{Proof of Lemma~\ref{ltrace}}\label{sec-6} In this section we prove the local trace inequality and the local inverse trace inequality in Lemma~\ref{ltrace}.
The proof is divided into two cases, the two dimensional case and
the three dimensional case. Recall that, $e\in\I_h$ is an interface segment/patch, $K^e\in\M_h$ is  an element containing $e$, and $K^e_i=K^e\cap\Om_i, i=1, 2$.

\subsection{Two dimensional case}\label{ssec-6-1}
Since $\Ga$ is a $C^2$ boundary of $\Om_1$, it can be expressed as a
union of open subsegments $\Ga_i, i\in\La:=\set{1, 2, \cdots, N}$
such that each $\Ga_i$ is parameterized by some function
$x=\br_i(\xi)\in [C^2(I_i)]^2$ satisfying $\abs{\br_i'(\xi)}\neq 0$, where
$I_i$ is an open interval. Since $\Ga_i\cap\Ga_j$ is either
empty or an open subsegment of $\Ga$, $e$ is contained entirely
in some $\Ga_k$ if $h_{K^e}$ is small enough, and hence $e$ is
parameterized by $x=\br(\xi), \xi\in I$ with $\br(\xi)=\br_k(\xi)$
and $I\subset I_k$. It is clear that the length of $I$ $\lesssim
h_{K^e}$.
\begin{figure}[htp]
\centering\includegraphics[scale=.8]{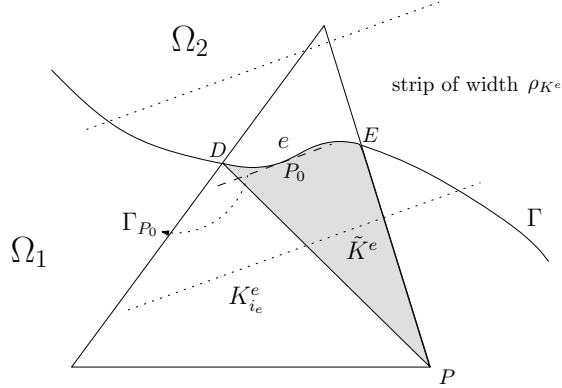} \caption{An
interface segment $e$ and the element $K^e$ in 2D}\label{interf2}\end{figure}

Let $P_0$ be any fixed point on $e$ and let $\Gamma_{P_0}$ be the tangent line to $\Gamma$ at $P_0$. Let $P$ be a point  in $K^e$ which achieves the maximum distance, denoted by $\tilde
h$, from the tangent line $\Gamma_{P_0}$. Denote by $\rho_{K^e}$ the diameter of the largest
disk contained in $K^e$. Clearly, we have $\tilde h \ge \rho_{K^e}/2 \gtrsim h_{K^e}$ (see Fig.~\ref{interf2},  otherwise $K^e$ would be
completely contained in a strip of width less than $\rho_{K^e}$ along
the axis $\Gamma_{P_0}$, which causes a contradiction). Now, we may choose $K^e_{i_e}$ to
be the one of two subregions $K^e_1$ and $K^e_2$ which contains the
point $P$ (in Fig.~\ref{interf2}, for instance, $K^e_{i_e}=K^e_1$).
Without loss of generality, we assume $P$ is the origin.  Suppose
$P_0$ corresponds to the parameter $\xi=\xi_0$. Then the vector $\overrightarrow{PP_0}=\br(\xi_0)$. 
It is clear that
\begin{equation}\label{r0}
\abs{\br(\xi_0)}\le h_{K^e},\quad \abs{\br'(\xi_0)}\eqsim 1.
\end{equation}
 We have the following
Taylor's expansions at $\xi_0$,
\begin{align}
\br(\xi)&=\br(\xi_0)+\br'(\xi_0)(\xi-\xi_0)+O(h_{K^e}^2) \label{r1},\\
\br'(\xi)&=\br'(\xi_0)+O(h_{K^e}). \label{r2}
\end{align}

To proceed, we need to bound the following term $G(\xi)$ from below:
$$G(\xi):=\frac{|\br(\xi)\times\br'(\xi)|}{|\br'(\xi)|},$$
 where $\mathbf{a}\times \mathbf{b}=a_{1}b_{2}-a_{2}b_{1}$ denotes the scalar cross product of two dimensional vectors. Recall that $|\mathbf{a}\times \mathbf{b}|=|\mathbf{a}||\mathbf{b}|\sin\langle\mathbf{a},\mathbf{b}\rangle$. Note that
$G(\xi_{0})=\displaystyle\frac{|\br(\xi_0)\times\br'(\xi_0)|}{|\br'(\xi_0)|}$
equals to the distance from $P$ to $\Gamma_{P_0}$, thus
$$G(\xi_{0})=\tilde h\gtrsim h_{K^e}.$$
By combining \eqref{r0}--\eqref{r2}, we derive that
$$G(\xi)=
\frac{\abs{\br(\xi_0)\times\br'(\xi_0)+O(h_{K^e}^{2})}}{\abs{\br'(\xi_0)+O(h_{K^e})}}=G(\xi_{0})+O(h_{K^e}^2).$$ Then
there exists a positive constant $h_0$ depending only on the interface $\Ga$ and the shape regularity of $K^e$, such that if $0<h_{K^e}\le h_0$, then
\begin{equation}
G(\xi)\gtrsim h_{K^e},
~~~\mbox{for all}~\xi\in I.\label{r3}
\end{equation}

 Next we prove the trace inequality \eqref{tr1}. Let $\tilde K^e$ be a fan-like region in $K^e_{i_e}$, with $e\cup
PD\cup PE$ as its boundary (cf. Fig.~\ref{interf2}), that is, 
$$\tilde K^e:=\set{x:\;x =t\cdot\br(\xi),\;t\in [0,1],\; \xi\in I}.$$
Let $v\in C^1(K^e_{i_e})$ and consider its restriction on $e$
as follows,
\begin{eqnarray} v^2(\br(\xi))&= &
\int_0^1\frac{\partial}{\partial t}\big((t^2v^2(t\cdot \br(\xi))\big)\rd
t=  \int_0^1 2(t^2vv_t+tv^2)\rd t.\label{est1}
\end{eqnarray}
From $v_t=\nabla v\cdot \br$, we have $$v^2(\br(\xi))\le  \int_0^1
2(|v||\nabla v||\br|+v^2)t\rd t.$$ Integrating along $e$, we
find
\begin{eqnarray}
\int_{e} v^2 \rd s&=&\int_{I} v^2(\br(\xi)) |\br'(\xi)|\rd \xi\label{est2}\\
&\le& 2\int_{I}\int_0^1(|v||\nabla
v||\br|+v^2)t |\br'(\xi)|\rd t \rd\xi\nonumber\\
&= & 2\int_{I}\int_0^1(|v||\nabla v||\br|+v^2)\cdot
t|\br(\xi)\times\br'(\xi)|\cdot\frac{1}{G(\xi)}\rd
t \rd \xi\nonumber\\
&\lesssim & \frac{1}{\inf_{\xi\in I}G(\xi)}\int_{I}\int_0^1(|v||\nabla v|h_{K^e}+v^2)\cdot
t|\br(\xi)\times\br'(\xi)|\rd t \rd \xi,\nonumber
\end{eqnarray}
where we have used $|\br(\xi)|\lesssim h_{K^e}$ to derive the last inequality.
By noting that $t|\br(\xi)\times\br'(\xi)|$ is the absolute value of  the Jacobian determinant
of the mapping $x =t\cdot\br(\xi)$ and using \eqref{r3}, we obtain that
\begin{align}\label{est2a}
\|v\|_{L^2(e)}^2
&\lesssim  \frac{1}{\inf_{\xi\in I}G(\xi)}\int_{\tilde K^e}(|v||\nabla v|h_{K^e}+v^2)\rd x\\
&\lesssim\|v\|_{L^2(\tilde K^e)}\|\nabla
v\|_{L^2(\tilde K^e)} +h_{K^e}^{-1}\|v\|_{L^2(\tilde K^e)}^2.\notag
\end{align}
Since
$\tilde K^e\subseteq K^e_{i_e}$, we conclude that
$$\|v\|_{L^2(e)}\lesssim h_{K^e}^{-\frac12}\|v\|_{L^2(K^e_{i_e})}+\|v\|_{L^2(K^e_{i_e})}^{\frac12}\|\nabla v\|_{L^2(K^e_{i_e})}^{\frac12}.$$ 
By a density argument, the above inequality is also
valid for $v\in H^1(K^e_{i_e})$. This completes the proof of
\eqref{tr1}.

It remains to prove the inverse trace inequality \eqref{tr2}. We need the following inverse inequality on $\mathcal{P}_{p}([0,1])$ (cf. \cite{schwab98}):
\begin{equation}\label{etr1d}\|w_{h}'\|_{L^2([0,1])}\lesssim p^2\|w_h\|_{L^2([0,1])}\quad \forall\, w_{h}\in \mathcal{P}_{p}([0,1]).
\end{equation}
By noting that $v_{h}(t\cdot\br(\xi))$ is a polynomial in $t$ of degree $\le p$ if $K^e$ is a triangle and of degree $\le 2p$ if $K^e$ is a parallelogram,
 we have from the above inverse inequality that
\begin{align}\label{est1a}
v_h^2(\br(\xi))&=\int_0^1\frac{\partial}{\partial t}\big(t^2v_h^2(t\cdot \br(\xi))\big)\rd
t=2\int_0^1t\,v_h(t\cdot \br(\xi))\frac{\partial}{\partial t}\big(t\,v_h(t\cdot \br(\xi))\big)\rd t\\
&\le 2\bigg(\int_0^1\big(t\,v_h(t\cdot \br(\xi))\big)^2 \bigg)^{\frac12}
\bigg(\int_0^1\Big(\frac{\partial}{\partial t}\big(t\,v_h(t\cdot \br(\xi))\big)\Big)^2 \bigg)^{\frac12}\notag\\
&\lesssim p^2\int_0^1 t^2\,v_h^2(t\cdot \br(\xi))\rd t.\notag
\end{align}
Following the same guidelines of \eqref{est2} and \eqref{est2a}, we derive that
\begin{align}\label{est3}
\|v_h\|_{L^2(e)}^2&=\int_{I} v_h^2(\br(\xi)) |\br'(\xi)|\rd \xi
\lesssim p^2\int_{I} \int_0^1 \,v_h^2(t\cdot \br(\xi))\, t\,|\br'(\xi)| \rd t \rd \xi\\
&\lesssim \frac{p^2}{\inf_{\xi\in I}G(\xi)}\int_{\tilde K^e} v_h^2 \rd
x\lesssim \frac{p^2}{h_{K^e}}\|v_h\|_{L^2(\tilde K^e)}^2,\notag
\end{align}
which yields the conclusion of \eqref{tr2}.

\subsection{Three dimensional case}\label{ssec-6-2}
The argument for the two dimensional case is readily
extended to the three dimensional one. We only sketch the
proof by indicating the necessary modifications in the proof for the two dimensional case.

Since $\Ga$ is a $C^2$ interface, it can be expressed as a
union of open subpatches $\Ga_i, i\in\La:=\set{1, 2, \cdots, N}$
such that each $\Ga_i$ is parameterized by some function
$x=\br_i(\xi,\eta)\in [C^2(U_i)]^3$ satisfying $\abs{\br_{i\xi}\times\br_{i\eta}}\neq 0$, where $U_i$ is an open domain in $\R^{2}$. Since $\Ga_i\cap\Ga_j$ is either
empty or an open subpatch of $\Ga$, $e$ is contained entirely
in some $\Ga_k$ if $h_{K^e}$ is small enough, and hence $e$ is
parameterized by $x=\br(\xi,\eta), (\xi,\eta)\in U$ with $\br(\xi,\eta)=\br_k(\xi,\eta)$
and $U\subset U_k$. It is clear that the diameter of $U$ $\lesssim
h_{K^e}$.
\begin{figure}[htp]
\centering\includegraphics[scale=.8]{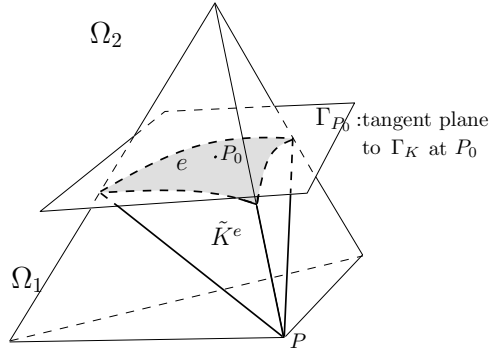} \caption{An
interface patch $e$ and the element $K^e$ in 3D}\label{interf3}\end{figure}

Let $P_0$ be any fixed point on $e$ and let $\Gamma_{P_0}$ be the tangent plane to $\Gamma$ at $P_0$. Then there exists a point $P\in K^e$ such that the distance from $P$ to $\Gamma_{P_0}$ $\gtrsim h_{K^e}$. Without loss of generality, we assume $P$ is the origin.  Suppose
$P_0$ corresponds to the parameter $(\xi,\eta)=(\xi_0,\eta_0)$. We have
\begin{align}
&\abs{\br(\xi_0,\eta_0)}\le h_{K^e},\quad \abs{\br_\xi(\xi_0,\eta_0)\times\br_\eta(\xi_0,\eta_0)}\eqsim 1. \label{r0-3}\\
&\br(\xi,\eta)=\br(\xi_0,\eta_0)+\br_\xi(\xi_0,\eta_0)(\xi-\xi_0)+\br_\eta(\xi_0,\eta_0)(\eta-\eta_0)+O(h_{K^e}^2) \label{r1-3},\\
&\br_\xi(\xi,\eta)\times\br_\eta(\xi,\eta)=\br_\xi(\xi_0,\eta_0)\times\br_\eta(\xi_0,\eta_0)+O(h_{K^e}). \label{r2-3}
\end{align}

Let 
\begin{equation*}
G(\xi,\eta):=\frac{\abs{\br(\xi,\eta)\cdot(\br_{\xi}(\xi,\eta)\times\br_{\eta}(\xi,\eta))}}
{\abs{\br_{\xi}(\xi,\eta)\times\br_{\eta}(\xi,\eta)}}.
\end{equation*}
By noting that $G(\xi_0,\eta_0)$ is the distance from $P$ to $\Gamma_{P_0}$, we have $G(\xi_0,\eta_0)\gtrsim h_{K^e}$. From \eqref{r0-3}--\eqref{r2-3},
there exists a positive constant $h_0$ depending only on the interface $\Ga$ and the shape regularity of $K^e$, such that if $0<h_{K^e}\le h_0$, then
\begin{equation}\label{r3-3}
G(\xi,\eta)\gtrsim h_{K^e}\quad \forall\, (\xi,\eta)\in U.
\end{equation}

Next we prove the trace inequality \eqref{tr1}. Let $\tilde K^e$ be the cone with apex $P$ and curved base $e$ as shown in Fig.~\ref{interf3}, i.e.,
$$\tilde K^e:=\set{x:\; x=t\cdot\br(\xi,\eta),\;t\in [0,1],\; (\xi,\eta)\in U}.$$
We have
\begin{equation} v^2(\br(\xi,\eta))= 
\int_0^1\frac{\partial}{\partial t}(t^3v^2(t\cdot \br(\xi,\eta))\rd
t=  \int_0^1 (2t^3vv_t+3t^2v^2)\rd t.\label{est1-3}
\end{equation}
Therefore, from $|\br|\lesssim h_{K^e}$ and the fact that $t^2\abs{\br(\xi,\eta)\cdot(\br_{\xi}(\xi,\eta)\times\br_{\eta}(\xi,\eta))}$ is the absolute value of  the Jacobian determinant
of the mapping $x=t\cdot\br(\xi,\eta)$, we conclude that
\begin{align}
&\int_{e} v^2 \rd \sigma=\int_{U} v^2(\br(\xi,\eta)) \abs{\br_\xi(\xi,\eta)\times\br_\eta(\xi,\eta)}\rd \xi\rd\eta \label{est2-3}\\
&\lesssim \int_{U}\int_0^1\big(|v||\nabla
v||\br|+v^2\big)\,t^2 \abs{\br_\xi(\xi,\eta)\times\br_\eta(\xi,\eta)}\rd t\rd \xi\rd\eta\nonumber\\
&\lesssim\int_{U}\int_0^1\big(|v||\nabla v|h_{K^e}+v^2\big)\,
t^2\abs{\br(\xi,\eta)\cdot(\br_{\xi}(\xi,\eta)\times\br_{\eta}(\xi,\eta))}\frac{1}{G(\xi,\eta)}\rd
t \rd \xi\rd\eta\nonumber\\
&\lesssim  \frac{1}{\inf_{(\xi,\eta)\in U}G(\xi,\eta)} \int_{\tilde K^e}(|v||\nabla v|h_{K^e}+v^2)\rd x\nonumber\\
&\lesssim\|v\|_{L^2(\tilde K^e)}\|\nabla v\|_{L^2(\tilde K^e)} + h_{K^e}^{-1}\|v\|_{L^2(\tilde K^e)}^2\,\nonumber
\end{align}
which implies that \eqref{tr1} holds.

It remains to prove \eqref{tr2}. By following the proof of \eqref{est1a}, we have
$$v_h^2(\br(\xi,\eta))\lesssim p^2\int_0^1 t^2\,v_h^2(t\cdot \br(\xi,\eta))\rd t.$$
Then, from \eqref{r3-3},
\begin{align}\label{est3-3}
\int_{e} v_h^2 \rd \sigma&=\int_{U} v_h^2(\br(\xi,\eta)) \abs{\br_\xi(\xi,\eta)\times\br_\eta(\xi,\eta)}\rd \xi\rd\eta \\
&\lesssim p^2\int_{U} \int_0^1 v_h^2(t\cdot \br(\xi,\eta))\, t^2\,\abs{\br_\xi(\xi,\eta)\times\br_\eta(\xi,\eta)}\rd t\rd \xi\rd\eta \notag\\
&\lesssim \frac{p^2}{\inf_{(\xi,\eta)\in U}G(\xi,\eta)} \int_{\tilde K^e}v_h^2\rd x
\lesssim \frac{p^2}{h_{K^e}}\norm{v_h}_{L^2(\tilde K^e)}^2,\notag
\end{align}
which implies that \eqref{tr2} holds. This completes the proof of Lemma~\ref{ltrace}.

\bibliographystyle{siam}
\bibliography{ref}
\end{document}